\documentclass[12pt,letterpaper]{article}

\usepackage{algpseudocode}
\usepackage{amscd}
\usepackage{amsmath}
\usepackage{amssymb}
\usepackage{amstext}
\usepackage{amsthm}
\usepackage{color}
\usepackage{graphics}
\usepackage{graphicx}
\usepackage{verbatim}
\usepackage{url}

\newtheorem{theorem}{Theorem}

\newtheorem{lemma}[theorem]{Lemma}
\newtheorem{corollary}[theorem]{Corollary}

\newtheorem{remark}[theorem]{Remark}

\date{}
\begin{document}

\title{\bf Cayley-Bacharach Formulas}\label{chapter:cayley_bacharach}
\author{Qingchun Ren, J\"urgen Richter-Gebert and Bernd Sturmfels}

\maketitle

\begin{abstract}
\noindent
The Cayley-Bacharach Theorem states that
all cubic curves through eight given points in the plane
also pass through a unique ninth point.
We write that point as an explicit rational function
in the other eight.
\end{abstract}
\smallskip

\section{Introduction}

This note concerns the following result from classical algebraic geometry.

\begin{theorem}[Cayley-Bacharach]
\label{thm:CB}
Let $P_1,\ldots,P_8$ be eight distinct points
in the plane, no three on a line, and no six on a conic.
There exists a unique ninth point $P_9$ such that
every cubic curve through $P_1,\ldots,P_8$
also contains~$P_9$. 
\end{theorem}

\vspace{-0.1in}
\[
\includegraphics[width=.5 \linewidth]{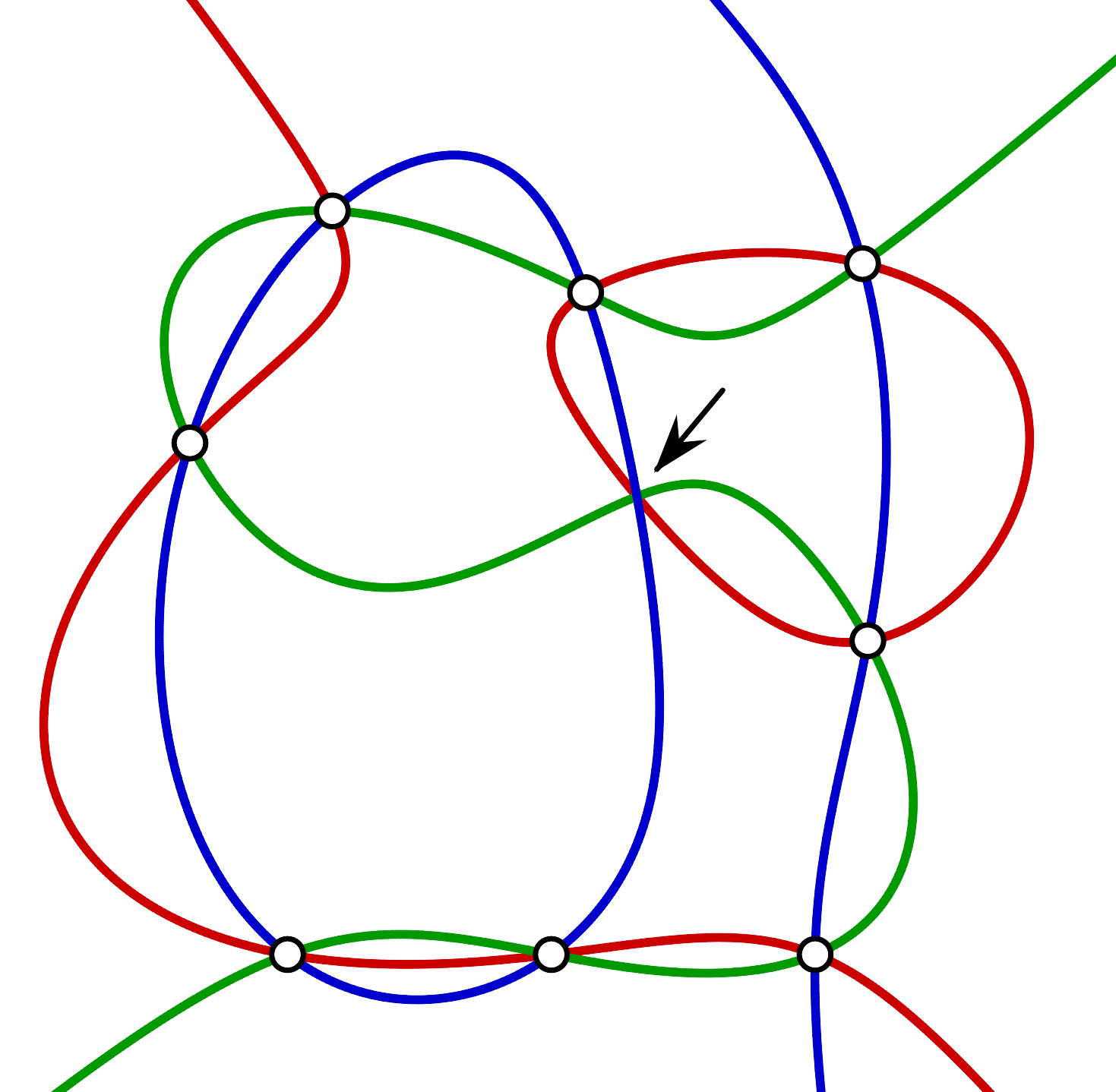}
\]
\vspace{-0.3in}
\[
\mbox{\footnotesize\it All cubics passing through the eight white points meet in a unique ninth point}
\]

This result refers to the projective plane $\mathbb{P}^2$.
It appears
 in most textbooks on plane algebraic curves.  For instance,
Kirwan asks for a proof in \cite[Exercise 3.13]{kirwan1992}. 
Theorem \ref{thm:CB} dates back to classical 19th century work of 
Hart~\cite{hart1851}, Weddle~\cite{weddle1851}, Chasles~\cite{chasles1853}, Cayley~\cite{cayley1862} and others.
While the 1851 articles of Hart and Weddle are mainly focused on
geometric constructions for the ninth point, 
Cayley's 1862 article is more algebraic and gives a complete proof.

In this paper we present explicit formulas for the
Cayley-Bacharach point in terms of algebraic invariants of the
  other eight points.  Our motivation arose
from computational projective geometry~\cite{RG2}.
The aim was to devise
numerically stable schemes for plotting $P_9$
when eight points $P_1,\ldots,P_8$ move
in animations of the Cayley-Bacharach Theorem created with {\tt Cinderella}~\cite{RG1}.
The formulas displayed in (\ref{equation:cayley_bacharach_formula}),
(\ref{equation:cayley_bacharach_simple}) and (\ref{eq:GGG})
are useful for that purpose.

In what follows we present our first formula.
In Section 2 we offer two proofs. The first exposits
Cayley's arguments in \cite{cayley1862}, while the second is a verification using
modern  computer algebra. In Section 3 we present our second formula.
That one is optimal with respect to degree and symmetry.
In Section 4 we close with a discussion on related issues and
further reading.

\smallskip

We write the  {\em Cayley-Bacharach point} $P_9$ 
as a rational expression in terms~of
$$ P_1=(x_1:y_1:z_1), \,P_2=(x_2:y_2:z_2),\, \dotsc{}, \,P_8=(x_8:y_8:z_8).  $$
 Such a formula exists because of the following argument.
   Consider the linear system of cubic curves through $P_1, P_2, \dotsc{}, P_8$. Its dimension is at least $\#\text{degrees of freedom}-\#\text{constraints}=10-8=2$. Choose two
distinct cubics  $C_1$ and $C_2$ in that system. Let
     $P_9 = (x_9:y_9:z_9)$ be their $9$th intersection point.
      In light of Theorem~\ref{thm:CB}, the point
      $P_9$ depends only on $P_1, P_2, \dotsc{}, P_8$.
From this one finds that
the Cayley-Bacharach Theorem  holds over any field. 

\begin{remark}\label{lemma:cayley_bacharach_rational_function}
The quotients $y_9/x_9$ and $z_9/x_9$ can be written as
rational functions in the $24$ unknowns $x_1,y_1,z_1,\dotsc{},y_8,z_8$.
The numerators and denominators of these rational functions
are polynomials with integer coefficients.
\end{remark}

We now define some polynomials that serve as ingredients in our formulas.
The condition for three points to lie on a line is the cubic polynomial
\begin{equation*}
[123] \,\, =\,\, \det
\begin{pmatrix}
x_1 & y_1 & z_1 \\
x_2 & y_2 & z_2 \\
x_3 & y_3 & z_3 \\
\end{pmatrix}.
\end{equation*}
The condition for six points to lie on a conic is given by the polynomial
\begin{equation}
\label{eq:Cbrack}
\begin{matrix}
C(P_1, P_2, \dotsc{}, P_6) & = & \det
\begin{pmatrix}
x_1^2 & x_1y_1 & x_1z_1 & y_1^2 & y_1z_1 & z_1^2 \\
x_2^2 & x_2y_2 & x_2z_2 & y_2^2 & y_2z_2 & z_2^2 \\
x_3^2 & x_3y_3 & x_3z_3 & y_3^2 & y_3z_3 & z_3^2 \\
x_4^2 & x_4y_4 & x_4z_4 & y_4^2 & y_4z_4 & z_4^2 \\
x_5^2 & x_5y_5 & x_5z_5 & y_5^2 & y_5z_5 & z_5^2 \\
x_6^2 & x_6y_6 & x_6z_6 & y_6^2 & y_6z_6 & z_6^2 \\
\end{pmatrix} 
\medskip \\
& = &  \,\,\, [123][145][246][356] - [124][135][236][456].
\end{matrix}
\end{equation}

Here is one more geometric condition of interest to us:
eight points lie on a  cubic curve that is singular at the first point.
This condition is expressed by  a polynomial
 of degree  $7 \cdot 3 + 3 \cdot 2 = 27$,  namely 
 $\,D(P_1; P_2, \dotsc{}, P_8) \,= $
  $$ \det
\begin{pmatrix} 
x_2^3 & x_2^2y_2 & x_2^2z_2 & x_2y_2^2 & x_2y_2z_2 & x_2z_2^2 & y_2^3 & y_2^2z_2 & y_2z_2^2 & z_2^3 \\
x_3^3 & x_3^2y_3 & x_3^2z_3 & x_3y_3^2 & x_3y_3z_3 & x_3z_3^2 & y_3^3 & y_3^2z_3 & y_3z_3^2 & z_3^3 \\
\dotsb{} & \dotsb{} & \dotsb{} & \dotsb{} & \dotsb{} & \dotsb{} & \dotsb{} & \dotsb{} & \dotsb{} & \dotsb{} \\
x_8^3 & x_8^2y_8 & x_8^2z_8 & x_8y_8^2 & x_8y_8z_8 & x_8z_8^2 & y_8^3 & y_8^2z_8 & y_8z_8^2 & z_8^3 \\
3x_1^2 & 2x_1y_1 & 2x_1z_1 & y_1^2 & y_1z_1 & z_1^2 & 0 & 0 & 0 & 0 \\
0 & x_1^2 & 0 & 2x_1y_1 & x_1z_1 & 0 & 3y_1^2 & 2y_1z_1 & z_1^2 & 0 \\
0 & 0 & x_1^2 & 0 & x_1y_1 & 2x_1z_1 & 0 & y_2^2 & 2y_1z_1 & 3z_1^2 \\
\end{pmatrix}.
$$
In these formulas we can change the indices.  For any
 $i,j,\ldots,k \in \{1,2,\ldots,8\}$, the expressions
$[ijk]$, $C(P_i, P_j, \dotsc{}, P_k)$ and $D(P_i; P_j, \dotsc{}, P_k)$
are well-defined homogeneous polynomials with  integer coefficients
in  $24 $ unknowns $x_i,y_j,z_k$.

To state the first main result of this note, we abbreviate
\begin{align*}
C_x &= C(P_1, P_4, P_5, P_6, P_7, P_8), & 
C_y &= C(P_2, P_4, P_5, P_6, P_7, P_8),\\
C_z &= C(P_3, P_4, P_5, P_6, P_7, P_8), &
D_x &= D(P_1; P_2, P_3, P_4, P_5, P_6, P_7, P_8),\\
D_y &= D(P_2; P_3, P_1, P_4, P_5, P_6, P_7, P_8),&
D_z &= D(P_3; P_1, P_2, P_4, P_5, P_6, P_7, P_8).
\end{align*}

\begin{theorem}\label{theorem:cayley_bacharach_formula}
The Cayley-Bacharach point is given by the formula
\begin{equation}\label{equation:cayley_bacharach_formula}
P_9 \,\, = \,\, C_xD_yD_z \cdot P_1 \,+\, D_xC_yD_z \cdot P_2 \,+\, D_xD_yC_z \cdot P_3.
\end{equation}
Equivalently, the coordinates of $P_9$ are the rational functions
\begin{equation}
\label{eq:x9y9z9}
\begin{matrix}
x_9 &=& C_xD_yD_zx_1 + D_xC_yD_zx_2 + D_xD_yC_zx_3, \\
y_9 &=& C_xD_yD_zy_1 + D_xC_yD_zy_2 + D_xD_yC_zy_3, \\
z_9 &=& C_xD_yD_zz_1 + D_xC_yD_zz_2 + D_xD_yC_zz_3.
\end{matrix}
\end{equation}
\end{theorem}

The following identity allows us
to write  the coefficients in (\ref{equation:cayley_bacharach_formula}) 
 in terms of the brackets $[ijk]$.
This can be verified using a computer algebra system.
\begin{small}
\begin{equation}
\label{eq:Dbrack}
\begin{matrix}
D(P_7;P_1,P_2,P_3,P_4,P_5,P_6,P_8) \,\,=  \qquad \qquad \qquad \qquad \\
\quad 3 \cdot ([647][857][478][128][173][423][573][526][176]\\
\qquad -[647][857][473][428][178][123][573][526][176]\\
 \qquad + [647][857][473][428][178][576][126][173][523]\\
\qquad + [657][847][573][528][178][123][473][426][176]\\
\qquad -[657][847][578][128][173][523][473][426][176]\\
\qquad \,\,\,\, -[657][847][573][528][178][476][126][173][423]).
\end{matrix}
\end{equation}
\end{small}

The following lemma is implied by the bracket expansions
in \eqref{eq:Cbrack} and \eqref{eq:Dbrack}.

\begin{lemma}\label{lemma-CDT}
Let $T$ be a projective transformation on $\mathbb{P}^2$, expressed as a $3 \times 3$ matrix
that acts on the homogeneous coordinates of the points $P_i$. Then
\begin{align*}
C(T(P_1), T(P_2), \dotsc{}, T(P_6)) &\,\,=\,\, \det{}(T)^4 \cdot C(P_1, P_2, \dotsc{}, P_6),\\
D(T(P_1); T(P_2), \dotsc{}, T(P_8)) &\,\,=\,\, \det{}(T)^9 \cdot D(P_1; P_2, \dotsc{}, P_8).
\end{align*}
\end{lemma}

In the next section we shall present two proofs of
Theorem \ref{theorem:cayley_bacharach_formula}.

\section{From Cayley to Computer Algebra}

In his 1862 paper~\cite{cayley1862}, Cayley describes a 
geometric construction for expressing $P_9$ rationally in $P_1,\ldots,P_8$. 
The key step is an implicit characterization of $P_9$
in terms of certain cross ratios. 
We set $[\![123456]\!]=C(P_1, P_2,\ldots, P_6)$ and 
\[
(1,2,3,4)_5:=
{
[513][524]
\over
[514][523]
}
\mbox{\quad and \quad}
(1,2,3,4)_{5678}:=
{
[\![567813]\!][\![567824]\!]
\over
[\![567814]\!][\![567823]\!]
}.
\]
The first expression is the cross ratio of the four lines spanned by $P_5$  and one of 
 $P_1,P_2,P_3$ or $P_4$. 
The second expression is the {\em cross ratio of four conics} passing through
 $P_5,P_6,P_7,P_8$ and one of the 
points  $P_1,P_2,P_3,P_4$.  
It is also the cross ratio of the four tangents at any of the intersection points.

\begin{proof}[First Proof of Theorem \ref{theorem:cayley_bacharach_formula}]
Cayley characterizes the point $P_9$ by the identity
\begin{equation}
\label{eq:cayley}
(5,6,7,8)_9 \,\,=\,\,(5,6,7,8)_{1234}.
\end{equation}
We shall prove this identify and then 
derive Theorem  \ref{theorem:cayley_bacharach_formula} from it. Let 
$C_{\lambda,\mu}=\lambda C_1+\mu C_2$ 
denote the pencil of conics through the points $P_1,P_2,P_3,P_4$,
and let  $L_{\lambda,\mu}=\lambda L_1+\mu L_2$ be the pencil of lines through an
auxiliary point~$X$. The intersection of these two pencils,
as $(\lambda:\mu)$ runs through $\mathbb{P}^1$,
is the cubic curve defined by $C_1 L_2 = C_2 L_1 $.
This cubic contains the points $P_1,P_2,P_3,P_4,X$. If we identify the pencil $C_{\lambda,\mu}$ with
$ \mathbb{P}^1$ via coordinates $(\lambda:\mu)$ 
 then  one can verify that the cross ratio of the four conics $C_{\lambda,\mu}$  through   $P_5,P_6,P_7,P_8$
equals $(5,6,7,8)_{1234}$. Similarly the cross ratio of the four lines in $L_{\lambda,\mu}$ through these points is
$(5,6,7,8)_X$.  Hence if  $P_5,P_6,P_7,P_8$ are chosen on the cubic
then $(5,6,7,8)_{1234}=(5,6,7,8)_X$.
Expanding this equation reveals that $X$ lies on a certain conic that passes through $P_5,P_6,P_7,P_8$. This conic is specified by the condition $(5,6,7,8)_X=l$, for some constant $l$. For each $X$ on this conic in general position, with proper choice of $C_1, C_2, L_1, L_2$, we recover a cubic that passes through $P_1,\ldots{},P_8$. In fact, it is the set of intersections of conics and lines that have identical cross ratios with respect to $P_5,P_6,P_7$ in the above sense. Every cubic that passes through $P_1,\ldots{},P_8$ arises this way. Note that this point-cubic correspondence depends only on $P_1,\ldots{},P_7$.

\[
\includegraphics[width=.6 \linewidth]{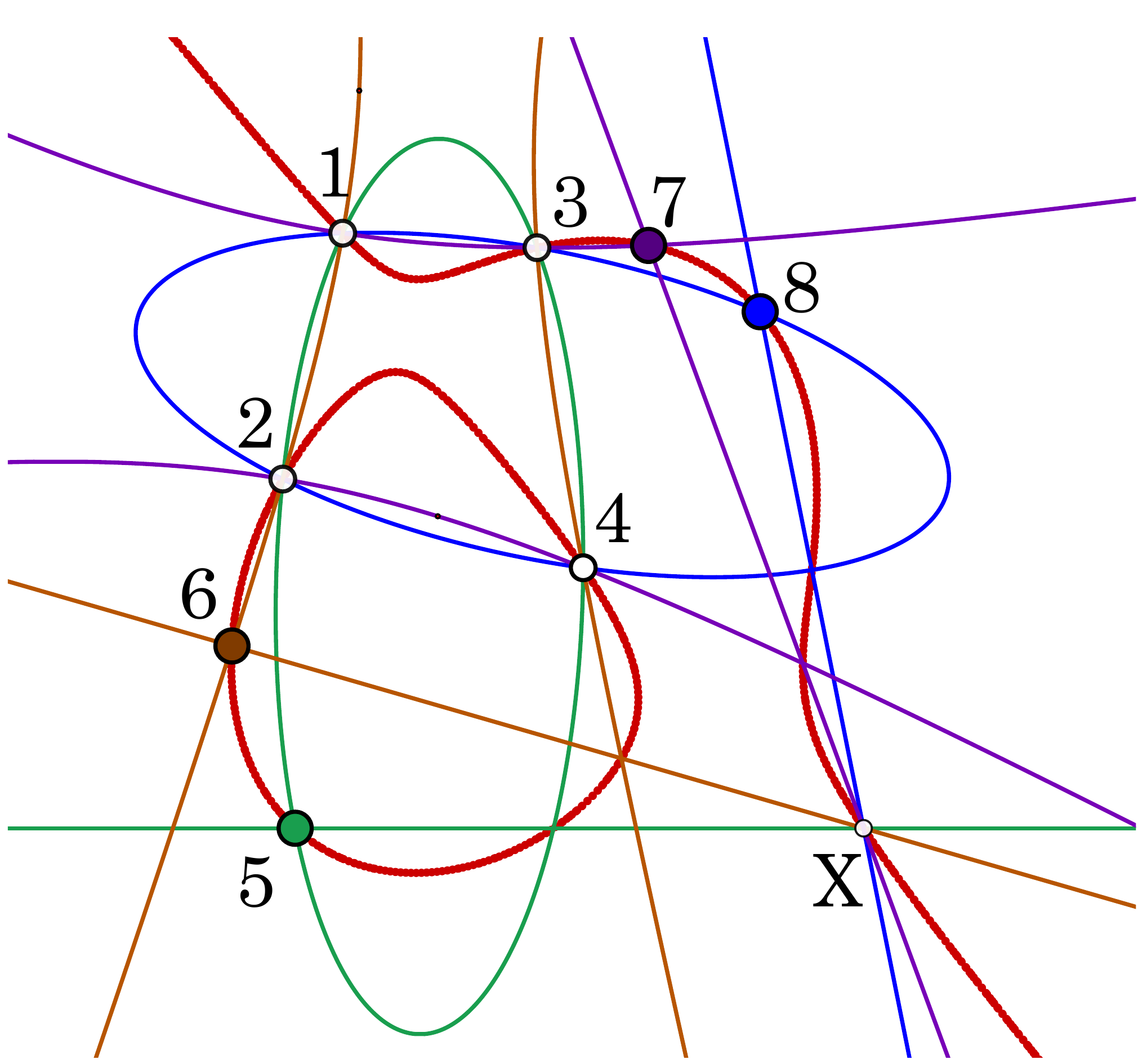}
\]
\[
\mbox{\footnotesize\it 
Construction of cubic curves. Corresponding lines and conics are drawn in the same color.}
\]

Consider the unique cubic through nine points $P_1, \ldots, P_9$ in general position.
It arises by applying the previous construction to any eight of them.
The corresponding point $X$ is in the intersection of the two conics $\mathcal{A} = \{ X :(5,6,7,8)_{1234}=(5,6,7,8)_X\}$ 
and  $\mathcal{B} = \{ X :(5,6,7,9)_{1234}=(5,6,7,9)_X\}$.
The other  intersections are $P_5,P_6,P_7$, so the point $X$ is uniquely specified.

Now assume that $P_9$ is the Cayley-Bacharach point of the other eight.
Then there is no unique cubic through $P_1,\ldots,P_9$. 
 The cubics passing through $P_1,\ldots{},P_8$ are exactly the same as the cubics passing through $P_1,\ldots{},P_7,P_9$.
In the sense of the above point-cubic correspondence, that means the two conics $\mathcal{A}$ and $\mathcal{B}$ must coincide. 
Hence $P_9$ lies on the conic $\mathcal{A} = \{ X \,:\,(5,6,7,8)_{1234}=(5,6,7,8)_X\}$.
We conclude that Cayley's condition (\ref{eq:cayley}) holds.

 \smallskip
 
 We next derive 
 Theorem~\ref{theorem:cayley_bacharach_formula} from (\ref{eq:cayley}).
 Suppose that $P_1,\ldots,P_8$ are given. By symmetry, the 
 Cayley-Bacharach point $P_9$ satisfies the two equations
 \begin{equation}
\label{eq:ccr}
(5,6,7,8)_9=(5,6,7,8)_{1234}=:l
\,\,\, \hbox{and} \,\,\,
(4,6,7,8)_9=(4,6,7,8)_{1235}=:m.
\end{equation}
Under the non-degeneracy assumption 
$[678]\neq0$, we can write $P_9\,=\,aP_6+bP_7+cP_8$.
We regard $(a:b:c)$ as homogeneous coordinates on  $\mathbb{P}^2$.
Inserting this expression for $P_9$ into $\,l = (5,6,7,8)_9\,$ creates the formula
\[l \,\,=\,\,
{ [957][968] \over [958][967] }
\,\,=\,\,
{ (a[657]+c[857])b[768] \over (a[658]+b[758])c[867] }.
\]
This can be simplified to
\[
    [657]\cdot ab \,+\,l[658]\cdot ac \,+\,(1-l)[857]\cdot bc\,\,=\,\,0.
\]
Similarly, inserting $P_9=aP_6+bP_7+cP_8$ into  $m = (4,6,7,8)_9$ leads to
\[
    [647] \cdot ab \,+\, m[648]\cdot ac \,+\,(1-m)[847] \cdot bc \,\,=\,\,0.
\]
These 
 two quadratic 
equations have four solutions in $\mathbb{P}^2$.
Three of them are $(1{:}0{:}0)$,
$(0{:}1{:}0)$ and $(0{:}0{:}1)$, corresponding to our basis points
$P_6$, $P_7$ and $P_8$. The fourth solution 
  $(a:b:c) $ gives the Cayley-Bacharach point $P_9$. It equals
\[\footnotesize
\begin{array}{ccc}
\Big(
(-[647] [857] (l-1) + [657] [847| (m-1)) ([658] [847] l (m-1) - [648] [857] (l-1) m) :\\[2mm]
 -([647] [658] l- [648] [657] m) ([658] [847] l (m-1) -     [648] [857] (l-1 ) m):\\[2mm] 
 -([647] [658] l -    [648] [657] m) ([647] [857] (l-1) + [657] [847] - [847] m))\Big).
\end{array}
\]
We now replace $l$ and $m$ in this expression by the right hand sides in
(\ref{eq:ccr}). After clearing denominators, expanding, dividing by common factors,
and rewriting bracket monomials, we arrive at
the formula (\ref{equation:cayley_bacharach_formula}) for $P_9$.
\end{proof}

Theorem~\ref{theorem:cayley_bacharach_formula}
can also be proved directly, by clever use  of computer algebra.

\begin{proof}[Second Proof of Theorem \ref{theorem:cayley_bacharach_formula}]
The ring $\mathbb{Z}[x_{\cdot{}},y_{\cdot{}},z_{\cdot{}}]$ is $\mathbb{Z}^8$-graded with $\mathrm{deg}(x_i)=\mathrm{deg}(y_i)$ $ =\mathrm{deg}(z_i)=\mathbf{e}_i$.
 The right hand side of (\ref{equation:cayley_bacharach_formula}) is a
 vector of homogeneous   polynomials of multidegree $(9,9,9,8,8,8,8,8)$. 
  By Lemma \ref{lemma-CDT}, it is equivariant
   under projective transformations on $\mathbb{P}^2$, up to a constant factor. 
 We may fix
 \begin{equation}
 \label{eq:eightpoints} \qquad
 \begin{matrix}
 P_1 = (1:0:0), \,\,P_2 = (0:1:0), \,\,P_3 = (0:0:1),  \,\, P_4 = (1:1:1), \\ 
 P_5 = (1:a:b),\,\,   P_6 = (1:c:d), \,\, P_7 = (1:e:f), \,\, P_8 = (1:g:h) . 
 \end{matrix} 
 \end{equation}
 If   (\ref{equation:cayley_bacharach_formula}) holds for
  such configurations of eight points  then it holds in general.

Let $u=y_9/x_9$ and $v=z_9/x_9$.  Since $P_1,P_2,\ldots,P_8,P_9$ lie on
two linearly independent cubics $C_1$ and $C_2$, the following matrix has rank at most $8$:
\begin{equation}\label{large_matrix}
\begin{pmatrix}
1    & 0    & 0    & 0    & 0    & 0    & 0    & 0    & 0    & 0    \\
0    & 0    & 0    & 0    & 0    & 0    & 1    & 0    & 0    & 0    \\
0    & 0    & 0    & 0    & 0    & 0    & 0    & 0    & 0    & 1    \\
1    & 1    & 1    & 1    & 1    & 1    & 1    & 1    & 1    & 1    \\
1    & a    & b    & a^2  & ab   & b^2  & a^3  & a^2b & ab^2 & b^3  \\
1    & c    & d    & c^2  & cd   & d^2  & c^3  & c^2d & cd^2 & d^3  \\
1    & e    & f    & e^2  & ef   & f^2  & e^3  & e^2f & ef^2 & f^3  \\
1    & g    & h    & g^2  & gh   & h^2  & g^3  & g^2h & gh^2 & h^3  \\
1    & u    & v    & u^2  & uv   & v^2  & u^3  & u^2v & uv^2 & v^3  \\
\end{pmatrix}
\end{equation}
Hence the $9 \times 9$-minors of (\ref{large_matrix}) are zero.
This gives $10$ equations in $u$ and $v$ whose coefficients are polynomials in $a, b, \dotsc{}, h$. Each equation is of the form
\begin{equation}\label{equation:u_v_relation}
A_1u^2v + A_2uv^2 + A_3u^2 + A_4uv + A_5v^2 + A_6u + A_7v \,\,=\,\, 0,
\end{equation}
where $A_1,A_2,\dotsc{},A_7\in{}\mathbb{Z}[a, b, c, d, e, f, g, h]$
are the cofactors in \eqref{large_matrix}.

For our special choices of $P_1,\ldots,P_8,P_9$, the formula in
 Theorem \ref{theorem:cayley_bacharach_formula} states 
 \begin{equation}
 \label{eq:assertion}
 u\,=\,\frac{D_xC_y}{C_xD_y} \quad \hbox{and} \quad v\,=\, \frac{D_xC_z}{C_xD_z} .
 \end{equation}
 To show this, we must argue that  \eqref{equation:u_v_relation} 
 holds after the substitution  (\ref{eq:assertion}). Equivalently,
 to prove Theorem \ref{theorem:cayley_bacharach_formula},
 we need to verify the $10$ identities of the form
 \begin{align*}
A_1C_y^2C_zD_x^2D_z+A_2C_yC_z^2D_x^2D_y+A_3C_xC_y^2D_xD_z^2+A_4C_xC_yC_zD_xD_yD_z \\ +A_5C_xC_z^2D_xD_y^2+A_6C_x^2C_yD_yD_z^2+A_7C_x^2C_zD_y^2D_z \,\,=\,\, 0.
\end{align*}
The left hand side lies in $\mathbb{Z}[a,b,c,d,e,f,g,h]$.
We will show that it is zero.

The computation needed to multiply out each term on the left hand side is still too large for a standard computer. A symbolic proof  using the computer algebra system
\texttt{sage}~\cite{sage} involves some tricks to
control intermediate expression growth as follows. 
Namely, we evaluate it in the following form:
\begin{align*}
\frac{C_zD_x\frac{A_1C_yD_z + A_2C_zD_y}{C_x} + A_3C_yD_z^2}{D_y} + A_4C_zD_z 
+ \frac{C_zD_y\frac{A_5C_zD_x + A_7C_xD_z}{C_y} + A_6C_xD_z^2}{D_x} .
\end{align*}
After computing $A_1C_yD_z + A_2C_zD_y$, one verifies that the result is divisible by $C_x$. Similarly, all other fractions in the above expression leave polynomial quotients. Therefore, the sizes of the intermediate results are limited.
\end{proof}
\smallskip

\section{Formula of Minimal Degree}

A natural question is whether the formula (\ref{equation:cayley_bacharach_formula})
is optimal in the sense that it has the lowest degree possible.
The answer is ``no''. We can do better.

The three polynomials in (\ref{eq:x9y9z9}) have multidegree $(9,9,9,8,8,8,8,8)$,
and they are all divisible by $[123]$. Removing that
common factor, we obtain three polynomials in $24$ unknowns
with  greatest common divisor~$1$.
The following statement can be verified with symbolic computations.
A theoretical proof was given in
 the PhD dissertation of the first author in \cite[Chapter 5]{QR}.
 
\begin{corollary}
The following formula for the Cayley-Bacharach point 
is invariant under the  symmetric group $S_8$
and contains no extraneous factor:
\begin{equation}\label{equation:cayley_bacharach_simple}
P_9 \,\,=\,\, \frac{1}{[123]} \cdot \bigl(C_xD_yD_z \cdot P_1 
\,+\, D_xC_yD_z \cdot P_2 \,+\, D_xD_yC_z \cdot P_3 \bigr)
\end{equation}
Its coordinates are homogeneous polynomials of degree $(8,8,8,8,8,8,8,8)$.
\end{corollary}

The expression \eqref{equation:cayley_bacharach_simple}
 is still not satisfactory because it involves division.
Our second main result is a highly symmetric
formula of optimal degree for $P_9$.
We shall use the following bracket monomial of
degree $(8,8,8,8,8,8,8,7)$:
\begin{equation}
\label{eq:FFF}
F(1,2,3,4,5,6,7;8) \,\,\,:=\,\,\,
\begin{matrix}
&\,\, [128] 
[238]
[348]
[458]
[568]
[678]
[718] \,
\\
& \cdot
[124]
[235]
[346]
[457]
[561]
[672]
[713] \,
\\
&\, \cdot
[126]
[237]
[341]
[452]
[563]
[674]
[715] .
\end{matrix}  \quad
\end{equation}

\begin{theorem}\label{FanoSum}
\label{theorem:jrg}  
The Cayley-Bacharach point $P_9$ is given by the formula
\begin{equation}
\label{eq:GGG}
\sum_{\pi\in S_8}
sign(\pi) \cdot
F(\pi(1),\ldots \pi(7);\pi(8)) \cdot P_{\pi(8)}.
\end{equation}
Here $\pi$ runs over all $40320$
 permutations in the symmetric group $S_{8}$.
\end{theorem}

Before addressing the validity of this formula,
we discuss how it was found.
We looked for a bracket expression of degree
$(8,8,8,8,8,8,8,8)$ that calculates $P_9$ in terms of the other eight points. The points $P_1,\ldots, P_8$ play a symmetric role in the calculation of $P_9$.
Switching any two of them should give the same result.
Furthermore if two points coincide then the formula
 should create the zero vector as an indication for degeneracy. 
 Thus we searched for a formula that was antisymmetric in $1,2,3,4,5,6,7,8$.
 This is ensured by the 
signed summation over $S_8$. Now let us focus on the structure of one summand.
We needed a product of $21$  brackets 
that is multiplied with the homogeneous coordinates of one of the points,
say $P_8$. In that  bracket monomial each other point must occur $8$ times while
$P_8$ occurs $7$ times.  Our  $F(1,2,3,4,5,6,7;8)$ reflects a particularly nice choice.
The first row  involves point $P_8$ seven times, while the pair $P_1 P_2$
is cyclically shifted.  A reasonable assumption is that the remaining 
 $14$ brackets are separated into two
{\em $7_3$ configurations}, i.e.~seven brackets 
with each point occurring    in exactly three triplets. Up to isomorphism there is only one
    $7_3$ configuration:
           {\it the Fano plane}. 
            There are $7!/168=30$ different ways to label a Fano plane.
  Among these precisely two are invariant under cyclically shifting the
indices   $1,\ldots, 7$. 
 These are precisely the Fano planes in the second and third row of (\ref{eq:FFF}):
  
\[
\includegraphics[width=.4\linewidth]{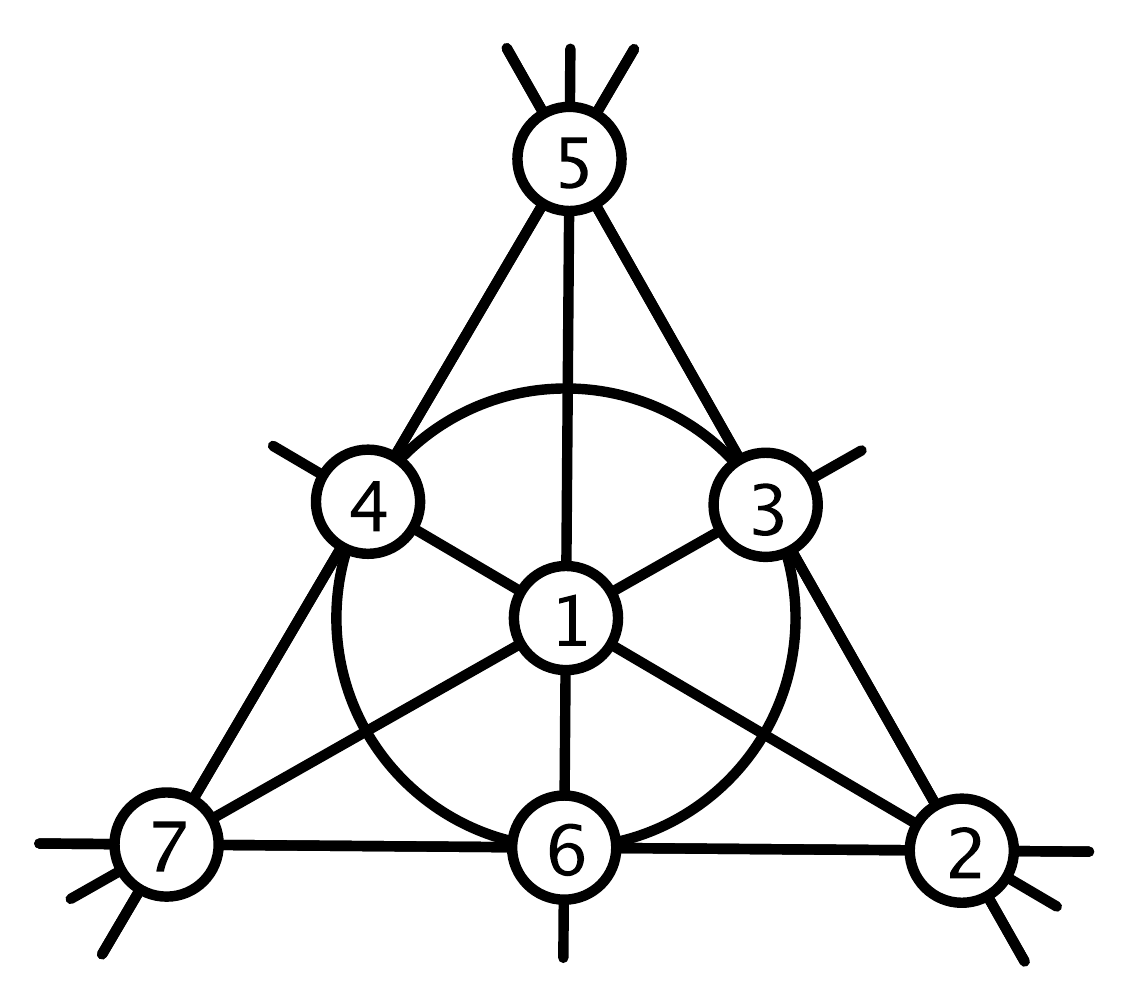}\qquad
\includegraphics[width=.4\linewidth]{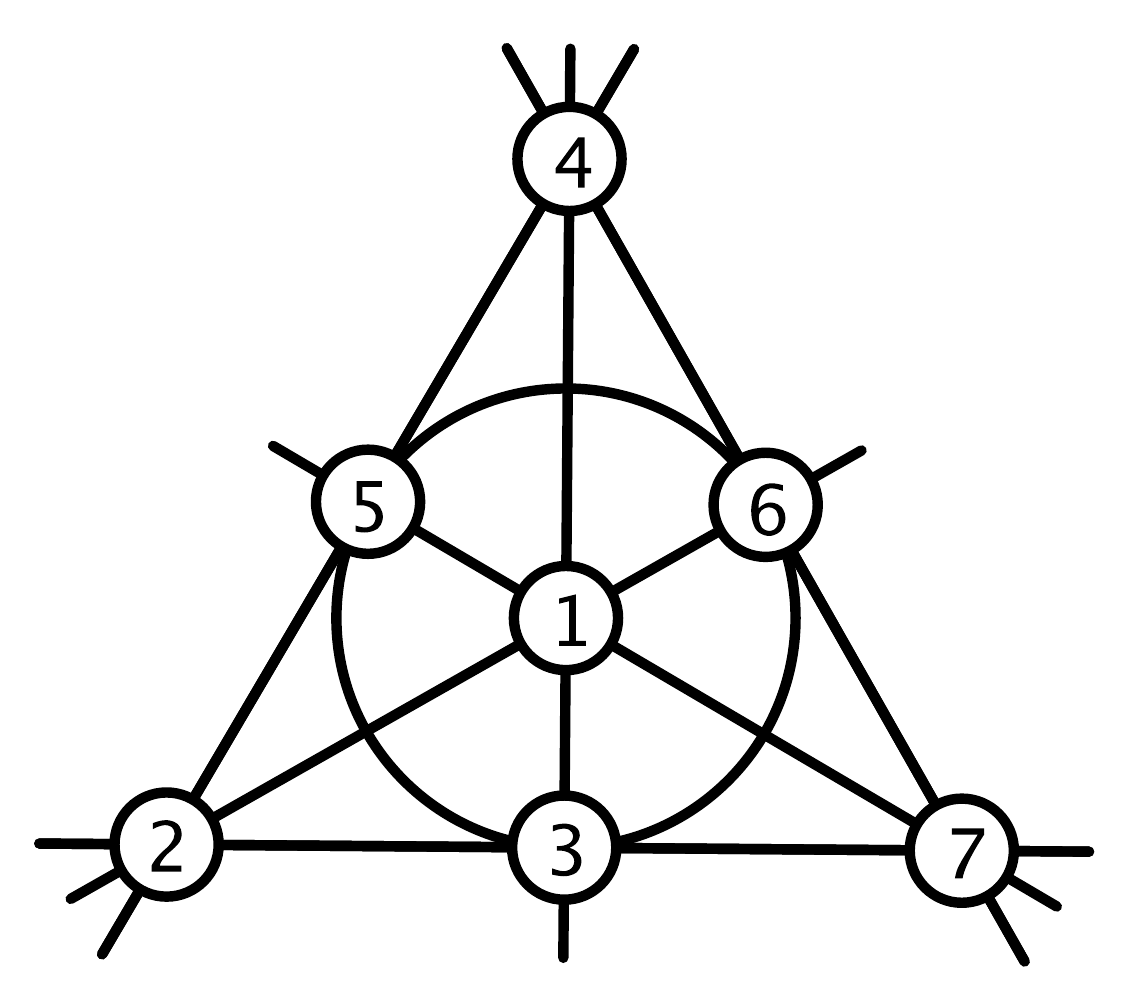}
 \]
 \vspace{-0.35in}
\[
\mbox{\footnotesize\it The two Fano planes appearing in $F(1,2,3,4,5,6,7;8)$.}
\]

\begin{proof}[Proof of Theorem \ref{FanoSum}]
The formula was verified using {\tt Mathematica} by comparing 
\eqref{eq:GGG} with the point $(x_9 \colon y_9 \colon z_9)$
created by the 
formula in Theorem \ref{theorem:cayley_bacharach_formula}. 
Let $(X_9:Y_9:Z_9)$ be the point calculated in \eqref{eq:GGG}.
 To prove Theorem \ref{FanoSum}, it is sufficient to show that 
 $x_9 Y_9= X_9 y_9$ and $x_9 Z_9=X_9 z_9$ for arbitrary choices of the points $P_1,\ldots, P_8$.
 It suffices to verify this for
the coordinates in (\ref{eq:eightpoints}).

   Strong confidence in our identities can be created
    by checking random
   specializations in exact arithmetic.
A brute force approach  by fully expanding (\ref{eq:GGG})
  ends up in combinatorial explosion because summing over $S_8$ creates 40320 terms.
However, one can apply the symmetries of the expression (\ref{eq:FFF})
  to significantly reduce the number of summands.
By cyclic shifting, we have
\[
F(1,2,3,4,5,6,7;8)\,\,=\,\, F(7,1,2,3,4,5,6;8).
\]
Replacing the cycle $1,2,\ldots, 7$ by its mirror image negates the expression:
\[
F(1,2,3,4,5,6,7;8)\,\,= \,\, -F(7,6,5,4,3,2,1;8).
\]
These two symmetries  allow us to perform the summation only over $2880=8!/14$ summands. 
With this simplification, 
we derived 
   a computer algebra proof of Theorem~\ref{FanoSum} using  {\tt Mathematica}.
   As in the proof of Theorem~\ref{theorem:cayley_bacharach_formula},
   we may assume that     (\ref{eq:eightpoints}) holds. 
     A straightforward simplification still ends in a combinatorial explosion.
     However the test can be carried out in approximately six hours if one variable is set to a fixed integer value.
     Since the degree of each variable is just 8, it suffices to perform this test of 9 different choices of this variable.
     This leads to a computer algebra proof, via {\tt Mathematica},      that runs for
     approximately two days on current standard hardware. \end{proof}

\section{Discussion}

Our contribution in this paper are two explicit formulas, 
in Theorems \ref{theorem:cayley_bacharach_formula}
and~\ref{theorem:jrg}, for the Cayley-Bacharach point $P_9$
in terms of eight given points in $\mathbb{P}^2$.
This adds to the geometric constructions known from the 
19th century literature.

A natural analogue to the Cayley-Bacharach Theorem
exists for eight points in $3$-space.
It states: {\em all quadric surfaces through seven given points in $\mathbb{P}^3$
also pass through a unique eighth point.}
The formula for that eighth is easier to derive than the one in Theorem
\ref{theorem:cayley_bacharach_formula}. It can be found in 
\cite[\S 7]{plaumann-sturmfels-vinzant2011}.
Both versions of the Cayley-Bacharach Theorem
play a prominent role in work of
 Blekherman \cite{Ble} on  sum of squares polynomials.
These are motivated by
recent advances in polynomial optimization.
Work of Iliman and De Wolff \cite[\S 3]{IW} suggests
 that our formulas will be useful 
in such domains of application.

Computing the Cayley-Bacharach point also makes sense
in {\em tropical geometry} \cite{MS}. In that setting,
all expressions in our formulas should be evaluated using arithmetic in the
min-plus semiring, with the determinant in the definitions of
$C_x,C_y,C_z,D_x,D_y,D_z$ replaced by the {\em tropical determinant}.
To assess the combinatorial structure and complexity
of the tropicalization of (\ref{equation:cayley_bacharach_simple}),
one examines the Newton polytopes of the numerators
and denominators. 

For example,
suppose $P_1=(1:0:0), P_2=(0:1:0), P_3=(0:0:1)$, and $ P_4=(1:1:1)$. Then
$P_9$ is given by the formula in (\ref{eq:assertion}). The factors
 $C_x,C_y,C_z,D_x,D_y,D_z$ are polynomials in $12$ variables $x_i,y_i,z_i$
 for $5\leq{}i\leq{}8$. 
  It can be verified with the software \texttt{polymake} \cite{polymake} that the
 six Newton polytopes are isomorphic. The f-vector for that common Newton polytope is $$(120,1980,7430,11470,8720,3460,700,60).$$ That is, the polytope has $120$ vertices, $1980$ edges,  and
  $60$ facets. The tropical polynomials
 ${\rm trop}(C_x), \ldots, {\rm trop}(D_z)$ are piecewise linear functions,
each given as the minimum of $120$ linear functions on $\mathbb{R}^{12}$.
From this, we obtain an explicit piecewise linear formula
for  ${\rm trop}(P_9)$  in terms of ${\rm trop}(P_1),\ldots,  {\rm trop}(P_8)$.
That formula is valid for scalars $x_i,y_i,z_i$ in a field
 with valuation, such as the $p$-adic numbers, provided
there is no cancelation of
lowest terms when evaluating  (\ref{equation:cayley_bacharach_simple}).
Unfortunately, cancellations  do occur in many situations,
and this topic deserves to be studied further.
We note that a
 {\em tropical Cayley-Bacharach Theorem} with weaker hypotheses
was given by Tabera in \cite{Tab}.

\smallskip

The Cayley-Bacharach Theorem offers students a friendly
point of entry into classical algebraic geometry \cite{Dolgachev, SempleRoth}. Those who use
computer algebra systems will appreciate our explicit formulas for  $P_9$
in terms of $P_1,\ldots,P_8$.
While the expressions  (\ref{equation:cayley_bacharach_formula}),
(\ref{equation:cayley_bacharach_simple}) and (\ref{eq:GGG}) seem to be new,
they rest on geometric constructions  that are very old and well known, notably from
\cite{cayley1862, chasles1853, hart1851, weddle1851, White}.

Here is one especially nice construction, related to {\em del Pezzo surfaces}.
Let $\mathcal{S}$ be the cubic surface in $\mathbb{P}^3$ that is obtained
by blowing up the plane $\mathbb{P}^2$ at the first six points
$P_1,\ldots,P_6$. Write $\tilde P_7$ and $\tilde P_8$
for the images on $\mathcal{S}$ of $P_7$ and $P_8$.
The line in $\mathbb{P}^3$ through $\tilde P_7$ and $\tilde P_8$ 
meets the cubic surface $\mathcal{S}$ in one other point 
$\tilde P_9$, namely the image in $\mathcal{S}$ of the
desired Cayley-Bacharach point $P_9$.

A referee kindly explained to us how Theorem  \ref{theorem:cayley_bacharach_formula}
can be derived from the {\em Geiser involution};
see \cite[Section 8.7.2]{Dolgachev} or \cite[Section 8.1]{SempleRoth}.
This is a Cremona transformation $G: \mathbb{P}^2 
\dashrightarrow \mathbb{P}^2$  given by 
seven fixed points $P_1,\ldots,P_7$.
Algebraic geometers should think of
fixing a marked del Pezzo surface of degree $2$.
The corresponding Geiser involution is the map $G$ that
takes $P_8$ to the Cayley-Bacharach point $P_9$. In coordinates, one can write
$G: (x{:}y{:}z) \mapsto (G_0(x,y,z) \! : \! G_1(x,y,z) \! : \! G_2(x,y,z))$
where $G_i$ are ternary forms of degree $8$ with
triple points at $P_1,\ldots,P_7$. The punchline is that our
$C_xD_yD_z$, $D_xC_yD_z$ and
$D_xD_yC_z$ are such polynomials of degree $8$ in the unknown
 $P_8 = (x:y:z)$.

\smallskip

In the literature, one can find numerous generalizations of Theorem~\ref{thm:CB}
that also carry the name ``Cayley-Bacharach''. To learn more about these, 
our readers might   start with the 1949 book of Semple and Roth \cite[Section V.1.1]{SempleRoth},
and then proceed to the 1996 article of Eisenbud, Green and Harris~\cite{eisenbud-green-harris1996}. 

\bigskip
\bigskip

{ \footnotesize
\noindent {\bf Acknowledgements}
The first and third author were supported by
the US National Science Foundation (DMS-0968882).
The second author was supported by
the  DFG Collaborative Research Center TRR 109, ``Discretization in Geometry and Dynamics''.
}

\bigskip


\begin{small}

\end{small}

\bigskip
\bigskip

\footnotesize
\noindent {\bf Authors' addresses:}

\smallskip

\noindent Qingchun Ren, Google Inc.,  Mountain View, CA 94043, USA, 
 {\tt qingchun.ren@gmail.com}
 
\medskip
\noindent J\"urgen Richter-Gebert, 
TU M\"unchen,
85747 Garching,
Germany,
{\tt richter@ma.tum.de}
\medskip

\noindent Bernd Sturmfels,  University of California, Berkeley, CA 94720,
USA, {\tt bernd@berkeley.edu}

\end{document}